\documentclass[11pt]{amsart}
\usepackage{amsmath}
\usepackage{amssymb}
\usepackage{amscd}

\def\NZQ{\mathbb}               
\def\NN{{\NZQ N}}

\def\ZZ{{\NZQ Z}}

\newtheorem{Theorem}{Theorem}[section]
\newtheorem{Lemma}[Theorem]{Lemma}

\newtheorem{Definition}[Theorem]{Definition}

\newtheorem{Question}[Theorem]{Question}
\let\epsilon\varepsilon
\let\phi=\varphi
\let\kappa=\varkappa

\begin{document}

\title{ Counterexamples to local  monomialization in positive characteristic}
\author{Steven Dale Cutkosky}
\thanks{Partially supported by NSF}

\address{Steven Dale Cutkosky, Department of Mathematics,
University of Missouri, Columbia, MO 65211, USA}
\email{cutkoskys@missouri.edu}


\maketitle

\section{Introduction}
In this paper we give counterexamples to local monomialization along a valuation in positive characteristic.

A fundamental problem in algebraic geometry is to find a general notion of simple morphism, and to then establish that it is possible to simplify a dominant morphism $\phi:X\rightarrow Y$ of varieties by performing  suitable monoidal transforms (blow ups of nonsingular subvarieties) above $X$ and $Y$ to obtain such a simple morphism.  

A local formulation of this problem can be given by asking for an appropriate simple form of the extension of local rings $\mathcal O_{Y,q}\rightarrow \mathcal O_{X,p}$ if $p\in X$ and $\phi(p)=q$. The notion of simplification after performing sequences of monoidal transforms above $\mathcal O_{Y,q}$ and $\mathcal O_{X,p}$ can then be expressed by fixing a valuation dominating $\mathcal O_{X,p}$. Then the valuation will dominate unique local rings on any sequences of monoidal transforms of $X$ and $Y$. 

A possible notion of a local simple morphism is that of monomial inclusion of regular local rings. We call such a map a monomial extension. Simplification by performing monoidal transforms is  called local monomialization. In characteristic zero, local monomialization is a good definition of simplification 
as a local monomialization can always be achieved along any valuation (\cite{C} and \cite{C2}).
In characteristic $p>0$, it is shown in \cite{CP} that local monomialization is true for generically finite maps of surfaces, if the extension of valuation rings is defectless (a condition which is always true in characteristic zero). It is further shown in \cite{C3} that this result holds for arbitrary defectless extensions of two dimensional regular local rings. 

It is natural to ask if local monomialization can always be achieved in positive characteristic. In this paper we give an example (Theorem \ref{TheoremNM2}) showing that this is not the case. The example even shows the failure of ``weak local monomialization'' in positive characteristic. Weak local monomialization holds if a monomial form can be reached by allowing arbitrary birational extensions instead of insisting on sequences of monoidal transforms. 

In the remainder of this paper, we will restrict to the case of generically finite morphisms.

\subsection{Local Monomialization.}


Suppose that $K^*/K$ is a finite separable field extension, $S$ is an excellent local ring of $K^*$ ($S$ has quotient field
$\mbox{QF}(S)=K^*$) and $R$ is an excellent  local ring of $K$ such that $\dim S=\dim R$, $S$  dominates $R$ ($R\subset S$ and the maximal ideals $m_S$ of $S$ and $m_R$ of $R$ satisfy  $m_S\cap R=m_R$) and $\nu^*$ is a valuation of $K^*$ which dominates $S$ (the valuation ring $V_{\nu^*}$ of $\nu^*$ dominates $S$).  Let $\nu$ be the restriction of $\nu^*$ to $K$.

The notation that we use in this paper  is explained  in more detail in Section \ref{SecNoc}.

\begin{Definition}\label{DefMon} $R\rightarrow S$ is {\it monomial} if 
$R$ and $S$ are regular local rings of the same dimension $n$ and  there exist regular systems of parameters $x_1,\ldots,x_n$ in $R$ and $y_1,\ldots,y_n$ in $S$, units $\delta_1,\ldots,\delta_n$ in $S$
and an $n\times n$ matrix $A=(a_{ij})$ of natural numbers with nonzero determinant such that
\begin{equation}\label{eqI1}
x_i=\delta_i\prod_{j=1}^ny_j^{a_{ij}}\mbox{ for $1\le i\le n$.}
\end{equation}
\end{Definition}

If $R$ and $S$ have equicharacteristic zero and algebraically closed residue fields, then within the extension $\hat R\rightarrow \hat S$ there are regular parameters giving a form (\ref{eqI1}) with all $\delta_i=1$.

More generally, we ask if a given extension $R\rightarrow S$ has a local  monomialization along the valuation $\nu^*$.

\begin{Definition}\label{DefI1} A {\it local monomialization} of $R\rightarrow S$ along the valuation $\nu^*$ is a commutative diagram
$$
\begin{array}{lllll}
R_1&\rightarrow &S_1&\subset& V_{\nu^*}\\
\uparrow&&\uparrow&&\\
R&\rightarrow&S&&
\end{array}
$$
such that the vertical arrows are products of monoidal transforms (local rings of blowups of regular primes)
and $R_1\rightarrow S_1$ is monomial.
\end{Definition}

It is proven in Theorem 1.1 \cite{C} that a local monomialization always exists when $K^*/K$ are algebraic function fields over a (not necessarily algebraically closed)
field $k$ of characteristic 0, and $R\rightarrow S$ are algebraic local rings of $K$ and $K^*$ respectively. (An algebraic local ring  is  essentially of finite type over $k$.)

We can also define the weaker notion of a {\it weak local monomialization} by only requiring that the conclusions of Definition
\ref{DefI1} hold with the vertical arrows  being required to be birational (and not necessarily factorizable by products of monoidal transforms).

This leads to the following question for extensions $R\rightarrow S$ as defined earlier in  this paper.

\begin{Question}\label{IQ} Does there always exist a local monomialization (or at least a weak local monomialization) of extensions $R\rightarrow S$ of excellent local rings dominated by a valuation $\nu^*$ ?
\end{Question} 

As commented above, the question has a positive answer within  algebraic function fields over an arbitrary field of characteristic zero by Theorem 1.1 \cite{C}.

For the question to have a positive answer in positive characteristic or mixed characteristic it is of course necessary that some form of resolution of singularities be true.
This is certainly true in equicharacteristic zero, and is 
known to be true very generally in dimension $\le 2$ (\cite{LU}, \cite{Li}, \cite{CJS}) and in positive characteristic and dimension 3 (\cite{RES}, \cite{C1}, \cite{CP1} and \cite{CP2}). A few recent papers going beyond dimension three are \cite{dJ}, \cite{Ha}, \cite{BrV}, \cite{BeV}, \cite{H1}, \cite{T1}, \cite{T2}, \cite{KK} and \cite{Te}.

The case of two dimensional algebraic function fields over an algebraically closed field of positive characteristic is considered in \cite{CP}, where it is shown that monomialization is true  if $R\rightarrow S$ is a {\it defectless} extension  of two dimensional algebraic local rings over an algebraically closed field $k$ of  characteristic $p>0$ (Theorem 7.3 and Theorem 7.35 \cite{CP}). This result is shown for defectless extensions of two dimensional regular local rings in \cite{C3}. We will discuss the important concept of  defect later on in this introduction.

In this paper we show that weak monomialization (and hence monomialization) does not exist in general for extensions of  algebraic local rings of dimension $\ge 2$ over a field  $k$ of  char $p>0$, giving a negative answer to Question \ref{IQ}. We prove the following theorem in Section \ref{SecCE}:

\begin{Theorem}\label{TheoremNM2} (Counterexample to local and weak local monomialization) Let $k$ be a field of characteristic $p>0$ with at least 3 elements and let $n\ge 2$. Then
there exists a finite separable extension $K^*/K$ of  n dimensional  function fields over $k$,  a valuation $\nu^*$ of $K^*$ with restriction $\nu$ to $K$ and algebraic regular local rings $A$ and $B$ of $K$ and $K^*$ respectively, such that $B$ dominates $A$, $\nu^*$ dominates $B$ and there do not exist
 regular algebraic local rings $A'$ of $K$ and $B'$ of $K^*$ such that $\nu^*$ dominates $B'$, $B'$ dominates $A'$, $A'$ dominates $A$, $B'$ dominates $B$ and $A'\rightarrow B'$ is monomial.
\end{Theorem}

We have that the defect $\delta(\nu^*/\nu)=2$ in the example of Theorem \ref{TheoremNM2} (with $\nu=\nu^*|K$).

In \cite{C} and \cite{CP}, a very strong form of local monomialization is established within characteristic zero algebraic function fields which we call strong local monomialization (Theorem 5.1 \cite{C} and Theorem 48 \cite{CP}).
This form is stable under appropriate sequences of monoidal transforms and encodes the classical invariants of the extension of valuation rings. In \cite{CP}, we show that strong local monomialization is true for defectless extensions of two dimensional algebraic function fields
(Theorem 7.3 and Theorem 7.35 \cite{CP}). This result is extended in \cite{C3} to defectless extensions of 2 dimensional regular local rings.
We give an example in \cite{CP} (Theorem 7.38 \cite{CP}) showing that strong local monomialization is not generally true for defect extensions of two dimensional algebraic function  fields (over a field of positive characteristic). However, local monomialization is true for this example, as is shown in Theorem 7.38, \cite{CP}.

We now define the defect of an extension of valuations. The role of this concept in local uniformization was  observed by Kuhlmann \cite{Ku1} and \cite{Ku2}.
A good introduction to the role of defect in valuation theory is given in \cite{Ku1}.
Its basic properties are developed  in  Section 11, Chapter VI \cite{ZS2},
\cite{Ku1} and
Section 7.1 of \cite{CP}.
Suppose that $K^*/K$ is a finite Galois extension of fields of characteristic $p>0$. The splitting field  $K^s(\nu^*/\nu)$ of $\nu$ is the smallest field $L$ between $K$ and   $K^*$ with the property  that $\nu^*$ is the only extension to $K^*$ of $\nu^*|L$. The defect $\delta(\nu^*/\nu)$ is defined by the identity
$$
[K^*:K^s(\nu^*/\nu)]=f(\nu^*/\nu)e(\nu^*/\nu)p^{\delta(\nu^*/\nu)}
$$
(Corollary to Theorem 25 , Section 12, Chapter VI \cite{ZS2}).
In the case when $K^*/K$ is only finite separable, we define the defect by
$$
\delta(\nu^*/\nu)=\delta(\nu'/\nu)-\delta(\nu'/\nu^*)
$$
where $\nu'$ is an extension of $\nu^*$ to a Galois closure $K'$ of $K^*$ over $K$.
The ramification index is
$$
e(\nu^*/\nu)=|\Gamma_{\nu^*}/\Gamma_{\nu}|
$$
and reduced degree is
$$
f(\nu^*/\nu)=[V_{\nu^*}/m_{\nu^*}:V_{\nu}/m_{\nu}].
$$

The defect is equal to zero if the residue field $V_{\nu}/m_{\nu}$  has characteristic zero (Theorem 24, Section 12, Chapter VI \cite{ZS2}) or if $V_{\nu}$ is a DVR (Corollary to Theorem 21, Section 9, Chapter V \cite{ZS1}).

Birational properties of two dimensional regular local rings along a valuation are generally the same in positive and mixed characteristic as in characteristic zero (this is illustrated in \cite{Ab1}, \cite{L2}, \cite{Sp} and \cite{CV1}), but properties related to ramification are very different, as illustrated by \cite{Ab5}, \cite{CP} and Theorem \ref{TheoremNM2} of this paper. However, for positive characteristic defectless extensions, the behavior in positive characteristic is similar to that of characteristic zero (\cite{CP}, \cite{C3}, \cite{GHK}, \cite{GK}).

We thank the referee for their careful reading of this  manuscript, and suggesting some improvements in the introduction.

\section{Notation and Preliminaries}\label{SecNoc}
\subsection{Local algebra.}  All rings will be commutative with identity. A ring $S$ is essentially of finite type over $R$ if $S$ is a local ring of a finitely generated $R$-algebra.
We will denote the maximal ideal of a local ring  $R$ by $m_R$, and the quotient field of a domain $R$ by $\mbox{QF}(R)$. 
(We do not require that a local ring be Noetherian). 
Suppose that $R\subset S$ is an inclusion of local rings. We will say that $S$ dominates $R$ if $m_S\cap R=m_R$.
If the local ring $R$ is a domain with $\mbox{QF}(R)=K$ then we will say that $R$ is a local ring of $K$. If $K$ is an algebraic function field over a field $k$ (which we do not assume to be algebraically closed) and a local ring $R$ of $K$ is essentially of finite type over $k$, then we say that $R$ is an algebraic local ring of $k$.  

Suppose that $K\rightarrow K^*$ is a finite field extension, $R$ is a local ring of $K$ and $S$ is a local ring of $K^*$. We will say that $S$ lies over $R$ if $S$ is a localization of the integral closure $T$ of $R$ in $K^*$. If $R$ is a local ring, $\hat R$ will denote the completion of $R$ by its maximal ideal $m_R$.

Suppose that $R$ is a regular local ring. A monoidal transform $R\rightarrow R_1$ of $R$ is a local ring of the form $R[\frac{P}{x}]_m$
where $P$ is a regular prime ideal in $R$ ($R/P$ is a regular local ring), $0\ne x\in P$   and $m$ is a prime ideal of $R[\frac{P}{x}]$ such that
$m\cap R=m_R$. $R_1$ is called a quadratic transform if $P=m_R$.

\subsection{Valuation Theory}
Suppose that $\nu$ is a valuation on a field $K$. We will denote by $V_{\nu}$ the valuation ring of $\nu$:
$$
V_{\nu}=\{f\in K\mid \nu(f)\ge 0\}.
$$
We will denote the value group of $\nu$ by $\Gamma_{\nu}$. Good treatments of valuation theory are Chapter VI of \cite{ZS2} and \cite{RTM}, which contain references to the original papers.
If $\nu$ is a valuation ring of an algebraic function field over a field $k$, we  insist that $\nu$ vanishes on $k\setminus\{0\}$,
and say that $\nu$ is a $k$ valuation.

If $\nu$ is a valuation of a field $K$ and $R$ is a local ring of $K$ we will say that $\nu$ dominates $R$ if the valuation ring
$V_{\nu}$ dominates $R$. Suppose that $\nu$ dominates $R$. A monoidal transform $R\rightarrow R_1$ is called a monoidal transform along $\nu$ if $\nu$ dominates $R_1$.

Suppose that $K^*/K$ is a finite separable extension, $\nu^*$ is a valuation of $K^*$ and $\nu$ is the restriction of $\nu$ to $K$.
The defect $\delta(\nu^*/\nu)$ is defined in the introduction to this paper. 
\subsection{Birational geometry of two dimensional regular local rings}
We recall some basic theorems  which we will make frequent use of.

\begin{Theorem}\label{ThmPre1}(Theorem  3 \cite{Ab1}) Suppose that $K$ is a field, and $R$ is a regular local ring of dimension two of $K$. Suppose that $S$ is another 2 dimensional regular local ring of $K$ which dominates $R$. Then there exists a unique sequence of quadratic transforms  
$$
R\rightarrow R_1\rightarrow \cdots \rightarrow R_n=S
$$
which factor $R\rightarrow S$.
\end{Theorem}

\begin{Lemma}\label{LemmaPre2}(Lemma 12 \cite{Ab1}) Suppose that $R$ is a two dimensional regular local ring of a field $K$ and $\nu$ is a valuation of $K$ which dominates $R$.
Let 
$$
R\rightarrow R_1\rightarrow R_2\rightarrow \cdots
$$
be the infinite sequence of quadratic transforms along $\nu$. 
Then
$$
V_{\nu^*}=\cup_{i=1}^{\infty}R_i.
$$
\end{Lemma}

We also make use of the fact that ``embedded resolution of singularities '' is true within a regular local ring of dimension 2 (Theorem 2 \cite{Ab1}),
and the fact that resolution of singularities is true for two dimensional excellent local rings (\cite{Li}, \cite{CJS}).

\section{Counterexamples to Local Monomialization in Positive Characteristic}\label{SecCE}

In this section we prove the counterexample Theorem \ref{TheoremNM2} to local and weak local monomialization stated in the introduction.

We first prove the theorem with the assumption that $n=2$.

Let $k$ be a field of characteristic $p>0$ containing at least three elements. 
Let $K^*$ be the two dimensional rational function field $K^*=k(x,y)$.  
Let 
$$
u=x^p(1+y), v=y^p+x,
$$
and let $K$ be the two dimensional rational function field $K=k(u,v)$. $K^*$ is separable over $K$ since the Jacobian of $u$ and $v$ is not zero. We have that $K^*=K(y)$ and $y$ satisfies the relation
$$
y^{p^2+1}+y^{p^2}-yv^p+(u-v^p)=0,
$$
so $K^*$ is a finite extension of $K$. Let $A=A_0=k[u,v]_{(u,v)}$ and $B=k[x,y]_{(x,y)}$. We have that $B$ dominates $A$.

\begin{Lemma}\label{LemmaNM1} Suppose that $A$ and $B$ are two dimensional regular local rings containing a common coefficient field $k$ of characteristic $p$ with at least three elements such that $B$ dominates $A$.
Let $u,v$ be regular parameters in $A$ and $x,y$ be regular parameters in $B$. Suppose that in $\hat B=k[[x,y]]$ we have expressions
\begin{equation}\label{eqNM1}
u=x^p(c_0+f_0y+x\Lambda_0),
v=\tau_0(y)y^p+e_0x+x\Omega_0
\end{equation}
where $c_0,f_0,e_0\in k$ are nonzero, $\tau_0(y)\in k[[y]]$ is a unit series, $\Lambda_0,\Omega_0\in \hat B$ and $\mbox{ord}_{\hat B}(\Omega_0)\ge 1$.
Let $\overline \tau_0\in k$ be the constant term of $\tau_0(y)$.
Suppose that $\alpha\in k$ is such that $\alpha\ne 0$ and $\alpha\ne -\frac{\overline\tau_0}{e_0}$.

Consider the sequence of quadratic transforms
$$
B=B_0\rightarrow B_1\rightarrow \cdots \rightarrow B_p
$$
where 
$B_i$ has regular parameters $x_i,y_i$ for $1\le i\le p$ defined by 
$$
x=x_iy_i^i, y=y_i\mbox{ if }1\le i<p
$$
and
$$
x=x_p^p(y_p+\alpha), y=x_p.
$$
Then $B_i$ does not dominate any quadratic transform of $A$ for $0\le i<p$. The sequence of quadratic transforms of $A$ which are dominated by $B_p$ are
$$
A=A_0\rightarrow A_1\rightarrow \cdots \rightarrow A_p,
$$
$A_i$ has regular parameters $u_i,v_i$, defined by
$$
u=u_iv_i^i, v=v_i\mbox{ if }1\le i<p,
$$
and
$$
u=u_p^p(v_p+\frac{\alpha^pc_0}{(\overline\tau_0+\alpha  e_0)^p}), v=u_p.
$$
In $\hat B_p=k[[x_p,y_p]]$ we have expressions
\begin{equation}\label{eqNM2}
u_p=x_p^p(c_p+f_py_p+x_p\Lambda_p),
v_p=\tau_p(y_p)y_p^p+e_px_p+x_p\Omega_p
\end{equation}
where $c_p,f_p,e_p\in k$ are nonzero, $\tau_p(y_p)\in k[[y_p]]$ is a unit series, $\Lambda_p,\Omega_p\in \hat B_p$ and $\mbox{ord}_{B_p}(\Omega_p)\ge 1$,
of the form of (\ref{eqNM1}).
Further, $A\rightarrow B_i$ is not monomial for $0\le i<p$.

\end{Lemma}

\begin{proof}  First suppose that $1\le i<p$. Then we have an expression
$$
\begin{array}{lll}
u&=&x_i^py_i^{pi}(c_0+f_0y_i+x_iy_i^i\Lambda_0),\\ 
v&=&\tau_0(y_i)y_i^p+e_0x_iy_i^i+x_iy_i^i\Omega_0=y_i^i(\tau_0(y_i)y_i^{p-i}+e_0x_i+x_i\Omega_0)
\end{array}
$$
in $\hat B_i=k[[x_i,y_i]]$. Thus no quadratic transform of $A$ factors through $A\rightarrow B_i$.

Now we consider factorization through $B_p$. We have an expression
$$
\begin{array}{lll}
u&=&x_p^{p^2}(y_p+\alpha)^p(c_0+f_0x_p+x_p^p(y_p+\alpha)\Lambda_0),\\ 
&=& x_p^{p^2}(y_p+\alpha)^p(c_0+f_0x_p+x_p^2\tilde\Lambda_0)\\
v&=&\tau_0(x_p)x_p^p+e_0x_p^p(y_p+\alpha)+x_p^p(y_p+\alpha)\Omega_0=x_p^p((\overline\tau_0+\alpha e_0)+e_0y_p+x_p\hat \Omega)
\end{array}
$$
in $\hat B_p=k[[x_p,y_p]]$, where $\tilde \Lambda, \hat \Omega\in B_p$. 
We  have an expression
$$
\frac{u}{v^p}= \left(\frac{y_p+\alpha}{(\overline\tau_0+\alpha e_0)+e_0 y_p+x_p\hat\Omega}\right)^p(c_0+f_0x_p+x_p^2\tilde \Lambda_0).
$$
There is an expression
$$
\frac{y_p+\alpha}{(\overline\tau_0+\alpha e_0)+e_0 y_p+x_p\hat\Omega}=\theta+\sigma(y_p)y_p+x_p\tilde \Omega
$$
in $\hat B_p=k[[x_p,y_p]]$, where 
$$
\theta=\frac{\alpha}{\overline\tau_0+\alpha e_0},
$$
$\sigma(y_p)\in k[[y_p]]$ is a unit series, and $\tilde\Omega\in \hat B_p$.
We thus have an expression
$$
\begin{array}{lll}
\frac{u}{v^p}&=& (\theta^p+\sigma(y_p)^py_p^p+x_p^p\tilde \Omega^p)(c_0+f_0x_p+x_p^2\tilde \Lambda_0)\\
&=& \theta^pc_0+c_0\sigma(y_p)^py_p^p+\theta^pf_0x_p+x_p\Omega_1
\end{array}
$$
with $\Omega_1\in \hat B_p$ and $\mbox{ord}_{\hat B_p}\Omega_1\ge 1$.

Thus the sequence of quadratic transforms of $A$ which are dominated by $B_p$ are
$$
A\rightarrow A_1\rightarrow A_2\rightarrow \cdots\rightarrow A_p
$$
where 
$A_i$ has regular parameters $u_i, v_i$ for $1\le i\le p$, defined by
$$
u=u_iv_i^i, v=v_i
$$
if $1\le i<p$, and
$$
u=u_p^p(v_p+\beta), v=u_p
$$
with 
$$
\beta=\theta^pc_0=\frac{\alpha^pc_0}{(\overline\tau_0+\alpha e_0)^p}.
$$
We have expressions 
$$
\begin{array}{lll}
u_p&=& x_p^p(c_1+f_1y_p+x_p\Lambda_1)\\
v_p&=& \tau_1(y_p)y_p^p+e_1x_p+x_p\Omega_1
\end{array}
$$
where
$$
c_1=\overline\tau_0+\alpha e_0, f_1=e_0, \Lambda_1=\hat\Omega,
$$
and
$$
\tau_1(y_p)=c_0\sigma(y_p)^p, e_1=\theta^pf_0.
$$

We now prove that the extensions $A\rightarrow B$ and $A\rightarrow B_i$ for $0\le i<p$ are not monomial.
 We will first establish this for the extension $A\rightarrow B$. By (\ref{eqNM1}), we have that $m_AB$ is not invertible and has order 1 in $B$.    We will suppose that $u',v'$ are regular parameters in $A$ which give a monomial form in $B$, and derive a contradiction. It is sufficient to show that a monomial form cannot exist in $\hat B=k[[x,y]]$. At least one of the $u',v'$ must have a nonzero $v$ term in its expansion in $\hat A=k[[u,v]]$; without loss of generality, we may assume that it is $v'$. By the Weierstrass Preparation Theorem
 (Corollary 1, page 145 \cite{ZS2}), we have an expression
 $$
 v'=\gamma(v-\psi(u))
 $$
 where $\gamma \in \hat A$ is a unit and $\psi(u)\in k[[u]]$. Thus $v'$ has an expression in $B$ of the same form as $v$, so we may assume that $v'=v$.
 
 For $F\in \hat B$, define the leading form of $F$ to be 
 $$
 L(F)=\sum_{i+j=r}a_{ij}x^iy^j
 $$
 if $\mbox{ord}_{\hat B}(F)=r$ and $F=\sum_{i+j\ge r}a_{ij}x^iy^j$. We have that $v$ is irreducible (and regular) in $\hat B$ with $L(v)=e_0x$.
 $u'$ must have a nonzero $u$ coefficient in its expansion in $\hat A$, so by the Weierstrass Preparation Theorem, 
 $$
 u'=\epsilon(u-\phi(v))
 $$
 where $\epsilon\in \hat A$ is a unit and $\phi(v)\in k[[v]]$.
 $v\not\,\mid u$ in $\hat B$ so $v\not\,\mid u'$ in $\hat B$. Since $u',v$ are assumed to give a monomial form in $B$, 
 we have that 
 \begin{equation}\label{eqF1}
 u'=\lambda v^iw^j
 \end{equation}
 where $\lambda$ is a unit in $\hat B$, and $w\in \hat B$ is such that $v,w$ is a regular system of parameters in $\hat B$. Since $m_AB$ is not invertible in $B$, we have that $i=0$ in (\ref{eqF1}). Thus $L(u')=\overline\lambda L(w)^j$ for some nonzero constant $\overline\lambda\in k$.
 Since $L(v)=e_0x$, we then have that $x\not\,\mid L(u')$ in $\hat B$. Thus $\phi(v)\ne 0$. Write
 $$
 \phi(v)=\sum_{i=r}^{\infty}a_iv^i
 $$
 where $a_i\in k$ for all $i$ and $a_r\ne 0$. For all $s\in \ZZ_+$, 
 $$
 v^s=e_0^sx^s+\mbox{terms of order greater than $s$}
 $$
 so
 $$
 L(\phi(v))=a_re_0^rx^r.
 $$
 Since $L(u)=c_0x^p$ and $x\not\,\mid L(u')$, we have that $r=p$ and $a_r=-\frac{c_0}{e_0^r}$. We have that
 $$
 \begin{array}{lll}
 u-\phi(v)&=& c_0x^p+f_0x^py+x^{p+1}\Lambda_0
 -\frac{c_0}{e_0^p}v^p-a_{p+1}e_0^{p+1}v^{p+1}+\mbox{terms of order $>p+1$ in $\hat B$}\\
 &=& f_0x^py+x^{p+1}\Lambda_0-a_{p+1}e_0^{p+1}x^{p+1}+\mbox{terms of order $>p+1$ in $\hat B$}.
 \end{array}
 $$ 
 Thus $x\mid L(u')$, a contradiction.
 
 We now show that $A\rightarrow B_i$ is not monomial for $1\le i<p$, by assuming that $u',v'$ are regular parameters in $A$ which have a monomial form in $B$, and deriving a contradiction. As in the previous case, we may assume that $v'=v$. We have a factorization
 $v=y_i^iw$ where $w$ is irreducible and regular, and 
 $$
 L(w)=\left\{\begin{array}{ll} e_0x_i&\mbox{ if $i<p-1$}\\
 \overline\tau_0y_i+e_0x_i&\mbox{ if $i=p-1$}.
 \end{array}\right.
 $$
 where $\overline\tau_0=\tau_0(0)$.
  Since $u',v$ give a monomial form in $B$, we must have an expression
 $$
 u'=\lambda y_i^aw^b
 $$
 where $\lambda$ is a unit in $B$. By Weierstrass Preparation, 
 \begin{equation}\label{eqNM10}
 u'=\epsilon(u-\phi(v))
 \end{equation}
 where $\epsilon$ is a unit in $\hat A$ and $\phi(v)\in k[[v]]$. We have that $w\not\,\mid u$ in $\hat B$ so $w\not\,\mid u'$. Thus $b=0$ and $\phi(v)\ne 0$. Further,
 \begin{equation}\label{eqNM11}
 L(u')=\lambda_0y_i^a
 \end{equation}
 where $\lambda_0$ is the constant term of $\lambda$.

 First assume that $i<p-1$. Now
 \begin{equation}\label{eqNM12}
 L(u)=c_0x_i^py_i^{pi}.
 \end{equation}
 Further,
\begin{equation}\label{eqNM13}
L(v)=e_0x_iy_i^i.
\end{equation}
Write
$$
\phi(v)=\sum_{i=r}^{\infty} a_iv^j\mbox{ where $a_r\ne 0$}.
$$
The only way (\ref{eqNM11}), (\ref{eqNM12}) and (\ref{eqNM13}) could be possible is if $r=p$. Then
$$
a_r=\frac{c_0}{e_0^p}
$$
and
$$
u-\phi(v)=f_0x_i^py_i^{pi+1}+\mbox{ terms of order $>pi+p+1$ in $\hat B$.}
$$
Thus $x_i\mid L(u'-\phi(v))$, a contradiction to (\ref{eqNM11}).

Now assume that $i=p-1$. Set
$$
\overline x=x_{p-1}, \overline y=y_{p-1}.
$$
We have that 
\begin{equation}\label{Ex1}
u-\phi(v)=\sigma \overline y^a
\end{equation}
where $\sigma$ is a unit series in $\overline x,\overline y$, whence
$$
L(u-\phi(v))=\sigma_0\overline y^a
$$
where $\sigma_0=\sigma(0,0)$. We have that 
$$
L(u)=c_0\overline x^p\overline y^{p(p-1)}
$$
and
$$
L(v)=\overline y^{p-1}(\overline\tau_0\overline y+\overline e_0\overline x).
$$
Writing
$$
\phi(v)=\sum_{i=r}^{\infty}a_iv^i
$$
with $a_r\ne 0$, we have that $r=p$ and
$$
a_p=\frac{c_0}{e_0^p}.
$$
Thus
\begin{equation}\label{Ex2}
\begin{array}{lll}
u-\phi(v)&=&
c_0\overline x^p\overline y^{p(p-1)}+f_0\overline x^p\overline y^{p(p-1)+1}
-\frac{c_0}{e_0^p}(\overline\tau_0\overline y^{p^2}+e_0^p\overline x^p\overline y^{p(p-1)})+\mbox{ terms of order $>p^2+1$}\\
&=& -\frac{c_0}{e_0^p}\overline\tau_0\overline y^{p^2}+f_0\overline x^p\overline y^{p^2-p+1}+\mbox{ terms of order $>p^2+1$.}
\end{array}
\end{equation}
Since 
$$
L(u-\phi(v))=-\frac{c_0}{e_0^p}\overline\tau_0\overline y^{p^2},
$$
we see from (\ref{Ex1}) that $a=p^2$ and that $\overline y^{p^2}$ divides $u-\phi(v)$ in $\hat B_{p-1}$; but this is impossible since $f_0\ne 0$ in (\ref{Ex2}).

\end{proof}

We now prove Theorem \ref{TheoremNM2}. We first establish the theorem when $n=2$.
By Lemma \ref{LemmaNM1} we can inductively construct infinite sequences of quadratic transforms
\begin{equation}\label{Ex3}
\begin{array}{cccccccc}
B&=&B_0&\rightarrow B_1\rightarrow \cdots \rightarrow &B_p&  \rightarrow B_{p+1}\rightarrow \cdots\rightarrow&B_{2p}&  \rightarrow \cdots\\
&&\uparrow&&\uparrow&&\uparrow&\\
A&=&A_0&\rightarrow A_1\rightarrow \cdots \rightarrow &A_p&\rightarrow A_{p+1}\rightarrow \cdots\rightarrow &A_{2p}& \rightarrow \cdots
\end{array}
\end{equation}
such that each $B_{pi+j}$ with $1\le j<p$ has regular parameters $x_{pi+j},y_{pi+j}$  such that
$$
x_{pi}=x_{pi+j}y_{pi+j}^j, y_{pi}=y_{pi+j}
$$
and $B_{p(i+1)}$ has regular parameters $x_{p(i+1)}, y_{p(i+1)}$ such that
 
$$
x_{pi}=x_{p(i+1)}^p(y_{p(i+1)}+\alpha_{p(i+1)}), y_{pi}=x_{p(i+1)}
$$
for some  $0\ne \alpha_{p(i+1)}\in k$.

Further, each $A_{pi+j}$ with $1\le j<p$, has regular parameters $u_{pi+j},v_{pi+j}$  such that
$$
u_{pi}=u_{p_i+j}v_{pi+j}^j, v_{pi}=v_{pi+j}
$$
and $A_{p(i+1)}$ has regular parameters $u_{p(i+1)}, v_{p(i+1)}$ such that 
$$
u_{pi}=u_{p(i+1)}^p(v_{p(i+1)}+\beta_{p(i+1)}), v_{pi}=u_{p(i+1)}
$$
for some $0\ne \beta_{p(i+1)}\in k$.
And for all $i$, we have expressions in $\hat B_{pi}=k[[x_{pi},y_{pi}]]$
\begin{equation}\label{eqNMi}
u_{pi}=x_{pi}^p(c_{i}+f_{i}y_{pi}+x_{pi}\Lambda_{i}),
v_{pi}=\tau_i(y_{pi})y_{pi}^p+e_ix_{pi}+x_{pi}\Omega_i
\end{equation}
where $0\ne c_i,f_i,e_i\in k$, $\tau_i(y_{pi})\in k[[y_{pi}]]$ is a unit series, $\Lambda_i,\Omega_i\in \hat B_{pi}$ and $\mbox{ord}_{\hat B_{pi}}(\Omega_i)\ge 1$.

Since $K$ and $K^*$ are two dimensional algebraic function fields, 
$V^*=\cup_{i\ge 0}B_i$ is a valuation ring of $K^*$ and $V=\cup_{i\ge 0}A_i$ is a valuation ring of $K$ (by Lemma \ref{LemmaPre2})) such that the valuation ring $V^*\cap K=V$. 
Let $\nu^*$ be the valuation of $K^*$ which satisfies $\nu^*(x)=1$ (where $x,y$ are our regular parameters in $B$) and whose valuation ring is $V^*$. If $f\in V^*$, then for $i>>0$, we have an expression
$f=\gamma x_i^m$ where $\gamma$ is a unit in $B_i$ and $m\in \NN$. From the relation
$$
\nu^*(y_{pi})=\nu^*(x_{p(i+1)})=\frac{1}{p}\nu^*(x_{pi}) 
$$
we obtain that 
$$
\Gamma_{\nu^*}=\cup \nu^*(x_i)\ZZ=\frac{1}{p^{\infty}}\ZZ.
$$
Similarly, the restriction of $\nu^*$ to $K$ is the valuation $\nu$ such that $\nu(u)=p$ and 
$$
\Gamma_{\nu}=\frac{1}{p^{\infty}}\ZZ=\Gamma_{\nu^*}.
$$

The top row of (\ref{Ex3}) is the complete list of all two dimensional regular algebraic local rings of $K^*$ dominating $B$ and dominated by $V^*$ (by Theorem \ref{ThmPre1}), and the bottom row of (\ref{Ex3}) is the complete list of all two dimensional regular algebraic local rings of $K$ dominating $A$ and dominated by $V$. 
If $A_i\rightarrow B_l$ is monomial, and $A_i\rightarrow A_j$ is a sequence of quadratic transforms which are dominated by $B_l$, then $A_j\rightarrow B_l$ is also monomial; we thus reduce to consideration of monomialization within the diagrams
$$
\begin{array}{cc}
B_{pi}&\rightarrow \cdots \rightarrow B_{p(i+1)-1}\\
\uparrow&\\
A_{p_i},&\end{array}
$$
 and the conclusions of  Theorem \ref{TheoremNM2} for the valuation $\nu^*$ on the extension $A\rightarrow B$ (in the case that $n=2$) then follow from  Lemma \ref{LemmaNM1}.

 We now  establish Theorem \ref{TheoremNM2} in the case that $n>2$. Let $r=n-2$ and  $t_1,\ldots,t_r$ be indeterminates. Let $K=k(t_1,\ldots,t_r,u,v)$ and $K^*=k(t_1,\ldots,t_r,x,y)$ with $u=x^p(1+y)$ and $v=y^p+x$. $K^*/K$ is a finite separable extension of $n$ dimensional algebraic function fields over $k$. Let
 $$
 A=k[t_1,\ldots,t_r,u,v]_{(u,v)}\mbox{ and }B=k[t_1,\ldots,t_r,x,y]_{(x,y)}.
 $$
 $A$ and $B$ are algebraic local rings of $K/k$ and $K^*/k$ respectively.
 Let $k'=k(t_1,\ldots,t_r)$. Then $K$ and $K^*$ are two dimensional algebraic function fields over $k'$. Further,
 $A\cong k'[u,v]_{(u,v)}$ and $B\cong k'[x,y]_{(x,y)}$ are algebraic local rings of $K/k'$ and $K^*/k'$ respectively.
 
 Let $\nu^*$ be the valuation of $K^*/K$ constructed in the $n=2$ case of the theorem which has the property that there do not exist regular algebraic local rings $A'$ of $K/k'$ and $B'$ of $K^*/k'$ such that $\nu^*$ dominates $B'$, $B'$ dominates $A'$ dominates $A$, $B'$ dominates $B$ and $A'\rightarrow B'$ is monomial. $\nu'$ is also a $k$ valuation of $K'$, and any algebraic local ring of $K^*/k$ (respectively $K/k$) which dominates $B$ (respectively $A$) must contain $k'$ so is also an algebraic local
  ring of $K^*/k'$ (respectively $K/k'$). Thus we have constructed an example satisfying the conclusions of Theorem \ref{TheoremNM2} for any $n\ge 2$.
 \vskip .2truein 
 
In \cite{CP}, we show that strong local monomialization is true for defectless extensions of two dimensional algebraic function fields
(Theorem 7.3 and Theorem 7.35 \cite{CP}). This result is extended in \cite{C3} to defectless extensions of 2 dimensional regular local rings. Thus the example constructed in Theorem \ref{TheoremNM2} must be a defect extension. We in fact have,  by (3) of Theorem 7.33 \cite{CP},  that the defect  $\delta(\nu^*/\nu)=2$. (The proof of Theorem 7.33 assumes that $k$ is algebraically closed, but the conclusions of the theorem are valid for arbitrary $k$ with the additional assumption that $V^*/m_{V^*}=k$,
which holds in our example).

\end{document}